\documentclass[english]{article}
\usepackage[T1]{fontenc}
\usepackage[latin9]{inputenc}
\usepackage{geometry}
\usepackage{float}
\usepackage{units}
\usepackage{mathrsfs}
\usepackage{amsthm}
\usepackage{amsmath}
\usepackage{amssymb}
\usepackage{graphicx}

\makeatletter
\newcommand{\lyxaddress}[1]{
\par {\raggedright #1
\vspace{1.4em}
\noindent\par}
}
\theoremstyle{plain}
\newtheorem{thm}{\protect\theoremname}
  \theoremstyle{plain}
  \newtheorem{lem}[thm]{\protect\lemmaname}

\makeatother

\usepackage{babel}
  \providecommand{\lemmaname}{Lemma}
\providecommand{\theoremname}{Theorem}

\begin{document}

\title{Orthocenters of triangles in the n-dimensional space}

\author{Wilson Pacheco (wpachecoredondo@gmail.com)\\
John Vargas (varjohn@gmail.com)\\
}

\maketitle

\lyxaddress{\begin{center}
Departamento de Matematicas\\
Facultad Experimental de Ciencias\\
Universidad del Zulia\\
Maracaibo - Venezuela
\par\end{center}}
\begin{abstract}
In this paper we present a way to define a set of orthocenters for
a triangle in the n-dimensional space $R^{n}$ and we will see some
analogies of these orthocenters with the classic orthocenter of a
triangle in the Euclidean plane. 
\end{abstract}

\section{Introduction}

In the Euclidean plane, the orthocenter $H$ of a triangle $\triangle ABC$
is defined as the point where the altitudes of the triangle converge,
i.e., the point at which lines perpendicular to the sides of the triangle
passing through the opposite vertex to such sides converge. If $O$
and $G$ are the circumcenter and the centroid of the triangle respectively,
the classical Euler's theorem asserts that $O$ , $G$ and $H$ are
collinear and $OG=2GH$.

Another property of the orthocenter of a triangle is the following:
the orthocenter is where concur the circles of radius equal to the
circumscribed passing through two vertices of the triangle, i.e if
the circumcircle is reflected with respect to the midpoints of the
sides of the triangle, then the three circles obtained concur in the
orthocenter of the triangle. Since this definition of orthocenter
not have to initially do with the notion of orthogonality we speak
in this case of $\mathcal{C}$-orthocenter. Moreover, by the definition
of $\mathcal{C}$-orthocenter of a triangle, it is the circumcenter
of the triangle whose vertices are the symmetric of the circumcenter
with respect to the midpoint of the sides.

If the triangle $\triangle ABC$ is not a right triangle, then the
triangles $\triangle HBC$ , $\triangle AHC$ and $\triangle ABH$
have the points A, B, C as orthocenters, respectively, i.e., triangles
with three vertices in the set$\left\{ A,B,C,H\right\} $ has as orthocenter
the remaining point. A set of four points satisfying the above property
is called an orthocentric system. Basic references to the orthocentric
system in Minkowski planes are in  \cite{M}.

When we review the properties related to the orthocenter such as;
Euler line, Feuerbach circumference, $\mathcal{C}$-orthocenter, orthocentric
system, we realize that their validity essentially depend on the relationship
between vertices and the circumcenter of the triangle, i.e, equidistance.
In this paper we will use this idea to define an 'orthocenter' associated
with each point that is equidistant from the vertices of a triangle
in the n-dimensional space and we will see some properties similar
to those of the orthocenter in the Euclidean plane.

\section{Notation and Preliminaries}

$R^{n}$ denote the classical $n$-dimensional Euclidean space, its
elements as vector space or affine space will call points and denoted
them with capital letters. if $A$ and $B$ are two points, then $\overrightarrow{AB}$
and $AB$ denote the vector and the standard segment with ends $A$
and $B$ respectively, i.e, $\overrightarrow{AB}=B-A\;\textrm{and }AB=\left\Vert B-A\right\Vert $.

A triangle $\triangle A_{0}A_{1}A_{2}$ is determined by three non
collinear points $A_{0}$ , $A_{1}$ and $A_{2}$ in the space $R^{n}$,
the points $A_{i}$ are called vertices of the triangle, the segment
denoted by $a_{i}$ whose endpoints are the vertices other than $A_{i}$
is called side of the triangle and is said to $A_{i}$ is his opposite
vertex. Denote by $O$, $\mathcal{C}$, $r$ and $G$ the circumcenter,
the circumcircle, the circumradius and the centroid of the triangle
$\triangle A_{0}A_{1}A_{2}$ , respectively, i.e, $O$ it is the only
point on plane determined by $A_{0}$ ,$A_{1}$, $A_{2}$ equidistant
from them, $\mathcal{C}$ is the circumference on afore mentioned
plane passing through $A_{0}$ ,$A_{1}$, $A_{2}$ , $r=OA_{0}=OA_{1}=OA_{2}$
and $G=\frac{A_{0}+A_{1}+A_{2}}{3}$, and $M_{i}$ the centroid (midpoint)
of side $a_{i}$. We also recall the medial or Feuerbach triangle
$\triangle M_{0}M_{1}M_{2}$ of the triangle $\triangle A_{0}A_{1}A_{2}$
, and denote its circumcenter by $Q_{O}$. Note that $Q_{O}=\frac{1}{2}\left(A_{0}+A_{1}+A_{2}-O\right)$.

If $P$ is a point of $R^{n}$ and $\lambda$ is a scalar, the homothetic
with center $P$ and ratio $\lambda$, is the application $\mathfrak{\mathscr{H}}_{P,\lambda}:R^{n}\rightarrow R^{n}$
defined by 
\[
\mathfrak{\mathscr{H}}_{P,\lambda}\left(X\right)=\left(1-\lambda\right)P+\lambda X,
\]
for all $X$ in $R^{n}$. $\mathsf{\mathscr{H}}_{P,-1}$ we will symbolize
by $\mathscr{\mathfrak{\mathit{\mathscr{S}}}_{P}}$ which is called
the point reflection with respect to $P$.

The following list contains some of the properties satisfied by the
orthocenter (see Figure 1). 

For the triangle $\triangle A_{0}A_{1}A_{2}$, the orthocenter $H$
is expressed as a function of the circumcenter $O$ and the vertices
of the triangle by the formula $H=A_{0}+A_{1}+A_{2}-2O$ and it is
not difficult to see that $H$ is the circumcenter of the triangle
$\triangle B_{0}{1}B_{-}B_{2}$, where $B_{i}$ is the symmetric of
$O$ with respect to $M_{i}$ for $i=0,1,2$, i.e, $B_{i}=A_{j}+A_{k}-O$.
The points $B_{0}$, $B_{1}$ and $B_{2}$ are the circumcenters of
the triangles $\triangle HA_{1}A_{2}$, $\triangle A0_{H}A_{2}$ and
$\triangle A_{0}A_{1}H$ respectively, the circumscribed circles of
these triangles are denoted by $\mathcal{C}_{0}$, $\mathcal{C}_{1}$
and $\mathcal{C}_{2}$, and all of them have radius $r$. The triangles
$\triangle A_{0}A_{1}A_{2}$ and $\triangle B_{0}B_{1}B_{2}$ are
symmetrical and in \cite{P-R} the triangle $\triangle B_{0}B_{1}B_{2}$
is called the antitriangle of the triangle $\triangle A_{0}A_{1}A_{2}$
associated with $O$. The center of symmetry between both triangles
is the point $Q_{O}$. 

\begin{center}
\begin{figure}[H]
\includegraphics{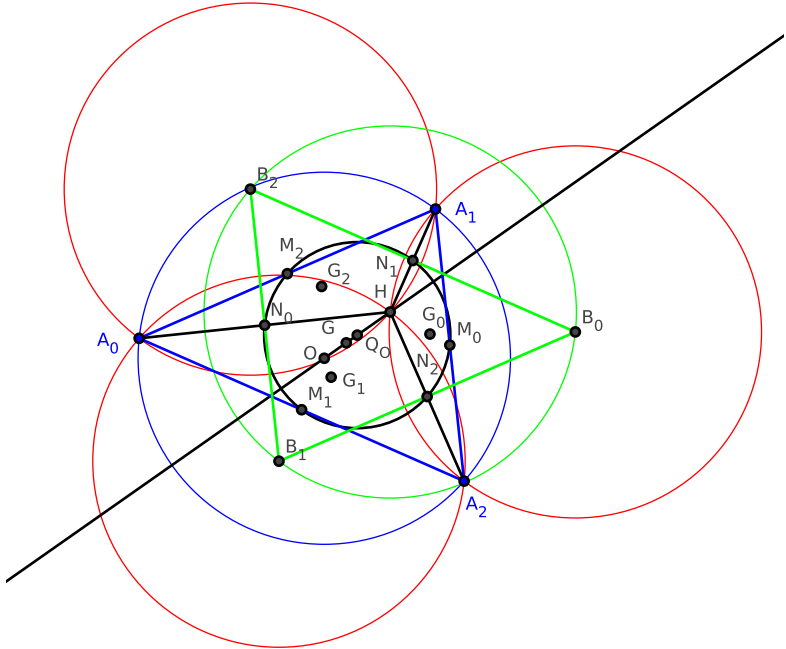}\caption{Orthocenter properties}
\end{figure}

\par\end{center}
\begin{enumerate}
\item The points $O$, $G$ and $H$ are collinear, with G in between, and
$2OG=GH$ ( Euler property).
\item If $N_{0}$, $N_{1}$ and $N_{2}$ are the midpoints of the sides
of the triangle $\triangle B_{0}B_{1}B_{2}$, the circumference of
center $Q_{O}$ and radio $\nicefrac{r}{2}$ (Feuerbach circumference)
passes through the points $M_{0}$, $M_{1}$,$M_{2}$ $N_{0}$, $N_{1}$
and $N_{2}$. It also passes through the midpoints of the segments
that joint $H$ with the points of the circumcircle of $\triangle A_{0}A_{1}A_{2}$.
and the midpoints of the segments that joint $O$ with the points
of the circumcircle of $\triangle B_{0}B_{1}B_{2}.$ 
\item the points $O$, $Q_{O}$, $G$, and $H$, form a harmonic range,
being satisfied $\frac{OG}{GQ_{o}}=\frac{OH}{HQ_{O}}=2$. 
\item The following sets are orthocentrics systems $\left\{ A_{0},A_{1},A_{2},H\right\} $,
$\left\{ B_{0},B_{1},B_{2},O\right\} $, $\left\{ M_{0},M_{1},M_{2},O\right\} $,
$\left\{ N_{0},N_{1},N_{2},H\right\} $ and $\left\{ G_{0},G_{1},G_{2},G\right\} $,
where $G_{0}$, $G_{1}$ and $G_{2}$ are the centroids of the triangles
$\triangle HA_{1}A_{2}$, $\triangle A_{0}HA_{2}$ y $\triangle A_{0}A_{1}H$
respectively.
\item If $\left\{ A_{0},A_{1},A_{2},A_{3}\right\} $ is an orthocentric
system, then $\overrightarrow{A_{i}A_{j}}\bot\overrightarrow{A_{k}A_{l}}$
for $\left\{ i,j,k,l\right\} =\left\{ 0,1,2,3\right\} $.
\end{enumerate}

\section{Results}

Given three non collinear points $A_{0}$, $A_{1}$, $A_{2}$ in the
Euclidean space, there is only one point that is equidistant from
them, which is precisely the circumcenter of the triangle $\triangle A_{0}A_{1}A_{2}$.
However, if the points $A_{0}$, $A_{1}$, $A_{2}$ are in an $n$-dimensional
space, with , then the set of equidistant points from $A_{0}$, $A_{1}$
and $A_{2}$ form an $(n-2)$-dimensional affine subspace, which we
denote by $\mathcal{\mathscr{\mathcal{\mathscr{C}}}}\left(\triangle A_{0}A_{1}A_{2}\right)$.
Each of these points is the center of an $n$-dimensional sphere passing
through the points $A_{0}$, $A_{1}$ and $A_{2}$. The following
theorem allows us to introduce the notion of an\"{ }orthocenter\textquoteright{}\textquoteright{}
associated with each point in $\mathcal{\mathscr{\mathcal{\mathscr{C}}}}\left(\triangle A_{0}A_{1}A_{2}\right)$,
and provides a generalization of the notion of C-orthocenter in the
plane. 
\begin{thm}
Let $\triangle A_{0}A_{1}A_{2}$ be a triangle in $R^{n}$ , $G$
its centroid and $H$ its orthocenter. If $P\in\mathcal{\mathcal{\mathscr{C}}}\left(\triangle A_{0}A_{1}A_{2}\right)$
and $\mathcal{S}$ is the sphere of center $P$ passing through the
points $A_{0}$, $A_{1}$ and A$_{2}$, and \textup{$r$ his} radius,
then the spheres $\mathcal{S}_{0}$, $\mathcal{S}_{1}$ and $\mathcal{S}_{2}$
that are symmetrical to $\mathcal{S}$ with respect to the midpoints
$M_{0}$, $M_{1}$ and $M_{2}$ of the sides of the triangle $\triangle A_{0}A_{1}A_{2}$
concur in the points $H$ and $H_{P}=A_{0}+A_{1}+A_{2}-2P$. Furthermore,
the following assertions hold: 
\begin{enumerate}
\item If $B_{0}$, $B_{1}$ and $B_{2}$ are the centers of $\mathcal{S}_{0}$,
$\mathcal{S}_{1}$ and $\mathcal{S}_{2}$ respectively, then the triangles
$\triangle A_{0}A_{1}A_{2}$ and $\triangle B_{0}B_{1}B_{2}$ are
symmetrical and the center of symmetry is the point $Q_{P}=\frac{1}{2}\left(A_{0}+A_{1}+A_{2}-P\right)$.
\item The points $P$, $G$ and $H_{P}$ are  collinear with $G$ between$P$
and $H_{P}$, and with $2PG=GH_{P}$. ( Euler property)
\item If $N_{0}$, $N_{1}$ and $N_{2}$ are the midpoints of the sides
of the triangle $\triangle B_{0}B_{1}B_{2}$, the sphere $\mathcal{S}_{M}$
of Center $Q_{P}$ and radio $\nicefrac{r}{2}$ passes through the
points $M_{0}$, $M_{1}$,$M_{2}$ $N_{0}$, $N_{1}$ and $N_{2}$.
It also passes through the midpoints of the segments that joint $H_{P}$
with the points of $\mathcal{S}$, and the midpoints of the segments
that joint $O$ with the points of the sphere $\mathcal{S}_{H}$ of
center $H_{P}$ and radius $r$ (Feuerbach sphere). 
\item The points $P$, $Q_{P}$, $G$, and $H_{P}$, form a harmonic range,
being satisfied $\frac{PG}{GQ_{P}}=\frac{PH_{P}}{H_{P}Q_{P}}=2$. 
\end{enumerate}
\end{thm}
\begin{proof}
Since the circumferences $\mathcal{C}_{0}$, $\mathcal{C}_{1}$ and
$\mathcal{C}_{2}$ are included in $\mathcal{S}_{0}$, $\mathcal{S}_{1}$
and $\mathcal{S}_{2}$ respectively, then the point $H$ is in the
spheres $\mathcal{S}_{0}$, $\mathcal{S}_{1}$ and $\mathcal{S}_{2}$,.
In order to see $H_{P}=A_{0}+A_{1}+A_{2}-2P$ is in $H_{P}=A_{0}+A_{1}+A_{2}-2P$,
it is enough to take a look at $H_{P}B_{i}=r,$ for $i=0,1,2$, where
$B_{I}$ is the center of $H_{P}B_{i}=r$. Note that $B_{i}=A_{J}+A_{k}-P$,
for $\left\{ i,j,k\right\} =\left\{ 0,1,2\right\} $. From which.

\[
H_{P}B_{i}=\left\Vert \left(A_{J}+A_{k}-P\right)-\left(A_{0}+A_{1}+A_{2}-2P\right)\right\Vert =\left\Vert P-A_{i}\right\Vert =r
\]

\begin{enumerate}
\item Note that $A_{i}+B_{i}=A_{i}+A_{J}+A_{k}-P$, where $\left\{ i,j,k\right\} =\left\{ 0,1,2\right\} .$
Therefore, the midpoint of $A_{i}B_{i}$ is $Q_{P}=\frac{1}{2}\left(A_{0}+A_{1}+A_{2}-P\right)$,
for $i=0,1,2$.
\item Since $2\left(G-P\right)=\frac{2}{3}\left(A_{0}+A_{1}+A_{2}-3P\right)=\frac{2}{3}\left(H_{P}-P\right)=H_{P}-G$,
it follow that $P$ , $G$ and $H_{P}$ are  collinear and $2PG=GH_{P}$.

\begin{figure}[H]
\includegraphics[scale=0.8]{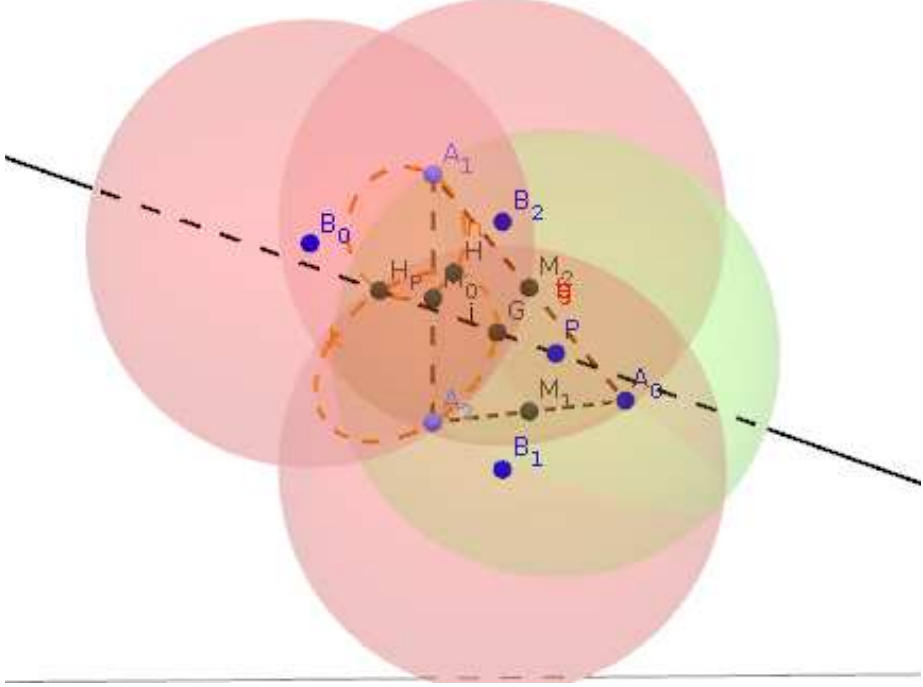}

\centering{}\caption{Orthocenter for a triangle in 3-dimesional space}
\end{figure}

\item By 1. we know that $\mathscr{S}_{Q_{P}}\left(\triangle A_{0}A_{1}A_{2}\right)=\triangle B_{0}B_{1}B_{2}$,
from where $\mathscr{S}_{Q_{P}}\left(M_{i}=N_{i}\right)$, for $i=0,1,2$,
For the first part only remains to show that $M_{i}Q_{P}=\nicefrac{r}{2}$,
for $i=0,1,2$. Indeed,
\begin{eqnarray*}
M_{i}Q_{P} & = & \left\Vert \frac{1}{2}\left(A_{0}+A_{1}+A_{2}-P\right)-\frac{1}{2}\left(A_{J}+A_{k}\right)\right\Vert \\
 & = & \frac{1}{2}\left\Vert A_{i}-P\right\Vert =\nicefrac{r}{2},
\end{eqnarray*}
for $i=0,1,2$.

\begin{figure}
\includegraphics[scale=0.8]{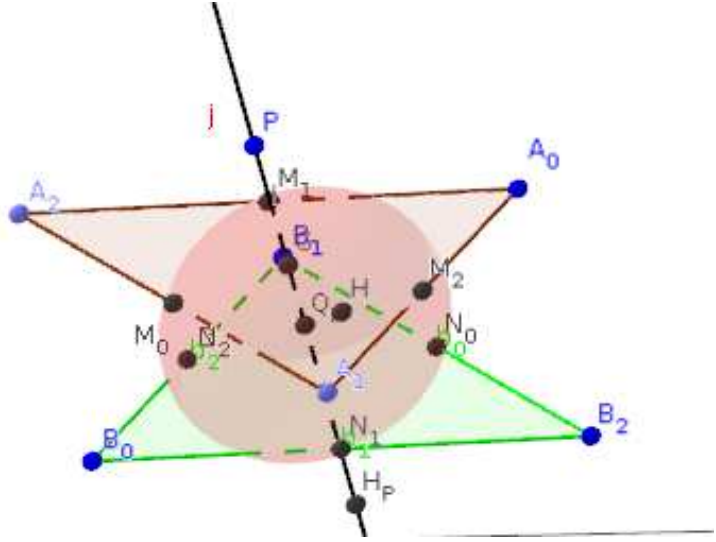}

\caption{Triangle, antitriangle and Feuerbach sphere}
\end{figure}

For the second part note that $\mathscr{H}_{H_{P},\nicefrac{1}{2}}\left(\mathcal{S}\right)=\mathcal{S}_{M}$
y $\mathscr{H}_{P,\nicefrac{1}{2}}\left(\mathcal{S}_{H}\right)=\mathcal{S}_{M}$
which implies the assertion.

\item Since 2 holds, $PG=\frac{1}{3}PH_{P}$. On the other hand
\begin{eqnarray*}
GQ_{P} & = & \left\Vert \left(\frac{1}{2}\left(A_{0}+A_{1}+A_{2}-P\right)-\frac{1}{3}\left(A_{0}+A_{1}+A_{2}\right)\right)\right\Vert \\
 & = & \frac{1}{6}\left\Vert \left(A_{0}+A_{1}+A_{2}-3P\right)\right\Vert =\frac{1}{6}PH_{P}
\end{eqnarray*}
and
\begin{eqnarray*}
Q_{P}H_{P} & = & \left\Vert \left(\left(A_{0}+A_{1}+A_{2}-2P\right)-\frac{1}{2}\left(A_{0}+A_{1}+A_{2}-P\right)\right)\right\Vert \\
 & = & \frac{1}{2}\left\Vert \left(A_{0}+A_{1}+A_{2}-3P\right)\right\Vert =\frac{1}{2}PH_{P},
\end{eqnarray*}
Finally, the assertion of the statement follows from the above relations.
\end{enumerate}
\end{proof}
We call the point $H_{P}$ the orthocenter of the triangle $\triangle A_{0}A_{1}A_{2}$
associated to $P$ and the set of all these orthocenters is denote
by $\mathcal{H}\left(\triangle A_{0}A_{1}A_{2}\right)$. The above
theorem says that the Euler property is satisfied, i.e., $\mathscr{H}_{G,\nicefrac{1}{2}}\left(\mathcal{C}\left(\triangle A_{0}A_{1}A_{2}\right)\right)=\mathcal{H}\left(\triangle A_{0}A_{1}A_{2}\right)$.
Furthermore, the orthocenter of the triangle $\triangle H_{P}A_{i}A_{j}$
associated to $B_{k}$ is the point $A_{k},$ where $\left\{ i,j,k\right\} =\left\{ 0,1,2\right\} $.
Thus, the notion of orthocentric system can be generalized to an n-dimensional
space, and we say that a set of four points $\left\{ A_{0},A_{1},A_{2},A_{3}\right\} $
is an orthocentric system, if there is a point $P\in\mathcal{C}\left(\triangle A_{0}A_{1}A_{2}\right)$
such that $A_{3}=A_{0}+A_{1}+A_{2}-2P$. We will see that the properties
about orthocentric systems in the plane previously listed are also
valid in this context. In fact, the following lemma is used for this
purpose..
\begin{lem}
The homothetic image of a $\mathcal{C}$-orthocentric system is a
$\mathcal{C}$-orthocentric system. 
\end{lem}
\begin{proof} Let $\{A_{0},A_{1},A_{2},A_{3}\}$ be a $\mathcal{C}$-orthocentric
system, then there exists $P_{4}\in\mathcal{C}(\triangle A_{0}A_{1}A_{2})$
such that $A_{3}=A_{0}+A_{1}+A_{2}-2P$. 

Let $B_{i}=\mathscr{H}_{C,\lambda}(A_{i})$, for $i=0,1,2,3$, and
$R=\mathscr{H}_{C,\lambda}(P)$. Clearly $R\in\mathcal{C}(\triangle B_{0}B_{1}B_{2})$
and 
\[
\begin{array}{lll}
B_{0}+B_{1}+B_{2}-2R & = & ((1-\lambda)C+\lambda A_{0})+((1-\lambda)C+\lambda A_{1})\\
 &  & +((1-\lambda)C+\lambda A_{2})-2((1-\lambda)C+\lambda P)\\
 & = & (1-\lambda)C+\lambda\left(A_{0}+A_{1}+A_{2}-2P\right)\\
 & = & (1-\lambda)C+\lambda A_{3}=B_{3},
\end{array}
\]
which completes the proof. \end{proof} 
\begin{thm}
Let $\triangle A_{0}A_{1}A_{2}$ be a triangle in $R^{n}$ , $G$
its centroid, $P\in\mathcal{\mathcal{C}}\left(\triangle A_{0}A_{1}A_{2}\right)$
and $H_{P}$ the orthocenter associated with $P$. Then the sets of
points $\left\{ A_{0},A_{1},A_{2},H_{P}\right\} $, $\left\{ B_{0},B_{1},B_{2},P\right\} $,
$\left\{ M_{0},M_{1},M_{2},O\right\} $\textup{, $\left\{ N_{0},N_{1},N_{2},H_{P}\right\} $}
and $\left\{ G_{0},G_{1},G_{2},G\right\} $ are orthocentric systems,
where $G_{0}$, $G_{1}$ and $G_{2}$ are the centroids of the triangles
$\triangle H_{P}A_{1}A_{2}$, $\triangle A_{0}H_{P}A_{2}$ y $\triangle A_{0}A_{1}H_{P}$
respectively.\end{thm}
\begin{proof}
We know that $M_{i}=\mathscr{H}_{G,-\nicefrac{1}{2}}(A_{i})$, for
$i=0,1,2,3$, and $\ \mathscr{H}_{G,-\nicefrac{1}{2}}(H_{P})=\frac{3}{2}G-\frac{1}{2}H_{P}=P$,
from which$\left\{ M_{0},M_{1},M_{2},P\right\} =\mathscr{H}_{G,-\nicefrac{1}{2}}\left(\left\{ A_{0},A_{1},A_{2},H_{P}\right\} \right)$

If $Q_{P}=\frac{1}{2}\left(A_{0}+A_{1}+A_{2}-P\right)$, then $\mathscr{S}_{\ensuremath{Q_{P}}}\left(P\right)=H_{P}$.
Thus, $\mathscr{S}_{\ensuremath{Q_{P}}}\left(\left\{ A_{0},A_{1},A_{2},H_{P}\right\} \right)=\left\{ H_{0},H_{1},H_{2},P\right\} $
and $\mathscr{S}_{\ensuremath{Q_{P}}}\left(\left\{ M_{0},M_{1},M_{2},P\right\} \right)=\left\{ N_{0},N_{1},N_{2},H_{P}\right\} $. 

Finally, $G_{i}=\frac{A_{i}+2A_{j}+2A_{k}-2P}{3}$, from which 
\[
\mathscr{H}_{Q_{P},-\nicefrac{1}{3}}\left(A_{i}\right)=\frac{4}{3}Q_{P}-\frac{1}{3}A_{i}=\frac{2}{3}\left(A_{0}+A_{1}+A_{2}-P\right)-\frac{1}{3}A_{i}=G_{i}
\]
and
\begin{eqnarray*}
\mathscr{H}_{Q_{P},-\nicefrac{1}{3}}\left(H_{P}\right) & = & \frac{4}{3}Q_{P}-\frac{1}{3}H_{P}=\frac{2}{3}\left(A_{0}+A_{1}+A_{2}-P\right)-\frac{1}{3}\left(A_{0}+A_{1}+A_{2}-2P\right)=G.
\end{eqnarray*}

Thus, $\mathscr{H}_{Q_{P},-\nicefrac{1}{3}}\left(\left\{ A_{0},A_{1},A_{2},H_{P}\right\} \right)=\left\{ G_{0},G_{1},G_{2},G\right\} $.\end{proof}
\begin{thm}
If \textup{$\left\{ A_{0},A_{1},A_{2},A_{3}\right\} $} is a orthocentric
system, then \textup{$\overrightarrow{A_{i}A_{j}}\bot\overrightarrow{A_{k}A_{l}}$
for $\left\{ i,j,k,l\right\} =\left\{ 0,1,2,3\right\} $}.\end{thm}
\begin{proof}
Since orthogonality in $R^{n}$ is equivalent to isosceles orthogonality,
we just need to see that$\left\Vert \overrightarrow{A_{i}A_{j}}-\overrightarrow{A_{k}A_{l}}\right\Vert =\left\Vert \overrightarrow{A_{i}A_{j}}+\overrightarrow{A_{k}A_{l}}\right\Vert $.

Indeed, consider the case; $i=0$, $j=1$, $k=2$, $l=3$, Let $P\in\mathcal{C}\left(\triangle A_{0}A_{1}A_{2}\right)$,
be such that $A_{3}=A_{0}+A_{1}+A_{2}-2P$ and $r$ be the radius
of the sphere with center $P$ passing through $A_{0}$, $A_{1}$
and $A_{2}$. Then
\[
\left\Vert \overrightarrow{A_{0}A_{1}}-\overrightarrow{A_{2}A_{3}}\right\Vert =\left\Vert \left(A_{1}-A_{0}\right)-\left(A_{3}-A_{2}\right)\right\Vert =\left\Vert 2\left(P-A_{0}\right)\right\Vert =2r
\]
and

\[
\left\Vert \overrightarrow{A_{0}A_{1}}+\overrightarrow{A_{2}A_{3}}\right\Vert =\left\Vert \left(A_{1}-A_{0}\right)+\left(A_{3}-A_{2}\right)\right\Vert =\left\Vert 2\left(A_{1}-P\right)\right\Vert =2r.
\]

The other cases are shown analogously.
\end{proof}
The above theorem tells us also that if $\left\{ A_{0},A_{1},A_{2},A_{3}\right\} $
is an orthocentric system and $A_{3}$ is not on the plane determined
$A_{0}$, $A_{1}$ and $A_{2}$. Then the tetrahedron $A_{0}A_{1}A_{2}A_{3}$
is an orthocentric tetrahedron, i.e, the altitudes of this tetrahedron
concur.

It is also important to note that, in our proof of the previous theorems
we do not use of the orthogonality properties in $R^{n}$ hence the
results presented here are still valid if we take any norm in $R^{n}$,
i.e., in Minkowsky spaces in general.

\end{document}